\documentclass[12pt]{article}
\newtheorem{theorem}{Theorem}[section]

\newtheorem{prop}[theorem]{Proposition}

\newtheorem{unnumber}{}

\newtheorem{prepf}[unnumber]{Proof}
\newenvironment{proof}{\prepf\rm}{\endprepf}

\newcommand{\per}{\mathop{\mathrm{per}}}

\begin{document}

\title{Hall's marriage theorem}
\author{Peter J. Cameron\\University of St Andrews\\
\texttt{pjc20@st-andrews.ac.uk}}
\date{}
\maketitle

\begin{abstract}
In 1935, Philip Hall published what is often referred to as ``Hall's marriage
theorem'' in a short paper (P.~Hall, On Representatives of Subsets,
\textit{J. Lond. Math. Soc.} (1) \textbf{10} (1935), no.1, 26--30.) This
paper has been very influential. I state the theorem and outline Hall's proof,
together with some equivalent (or stronger) earlier results, and proceed to
discuss some the many directions in combinatorics and beyond which this theorem
has influenced.
\medskip


\end{abstract}

\section{Introduction}

According to the zbMath database, Philip Hall wrote only 53 papers, a
relatively small number, 48 of which were single-authored. His name is known
to generations of group theorists because of his profound contributions to
the subject. But his most cited paper, with nearly twice as many citations
as the Hall--Higman paper in second place~\cite{hall_higman}, is the one in
which he proved what is now called \emph{Hall's marriage theorem}~\cite{hall_p}.
It is titled ``On representatives of subsets'', and is four and a
half pages long.

The application which Hall may have had in mind is the fact that, if $G$
is a group and $H$ a subgroup of finite index $n$, then there are $n$
elements $g_1,\ldots,g_n\in G$ which are simultaneously left and right coset
representatives of $H$ in $G$. However, John Britnell informs me that he and
Mark Wildon searched Hall's published work and found no evidence that he ever
used this result. Indeed, the result had been published by G.~A.~Miller in
1910~\cite{miller}.

Before stating the theorem, I give one definition, with a comment. Let
$(T_1,\ldots,T_n)$ be an $n$-tuple of subsets of a finite set $S$. A
\emph{system of distinct representatives} or SDR for the $n$-tuple is an
$n$-tuple $(a_1,\ldots,a_n)$ of elements of $S$ such that $a_i\in T_i$ for
all $i$, and $a_i\ne a_j$ for $i\ne j$. (Hall calls this a complete set of
distinct representatives, or CDR; to avoid confusion I have used the current
term.) Also, the set $\{a_1,\ldots,a_n\}$ is called a \emph{transversal} to
the family. 

\textit{Hall's Theorem} states:

\begin{theorem}
A necessary and sufficient condition for the existence of an SDR for a
tuple $(T_1,\ldots,T_n)$ of sets is that, for $0\le k\le n$, any $k$ of
the sets contain between them at least $k$ elements of $S$.
\end{theorem}

\begin{proof}
The necessity of the condition is clear. I sketch Hall's proof of the
sufficiency. It depends on a lemma asserting that,
with the hypotheses of the theorem, if $(a_1,\ldots,a_n)$ is an SDR and
\[P=\{i:1\le i\le n, a_i\hbox{ lies in every SDR}\},\]
then the union of the sets $T_i$ for $i\in P$ is $\{a_i:i\in P\}$.

Given this, the proof is by induction on $n$. The induction clearly starts
at $n=1$.

Suppose that the hypotheses of the theorem hold, and that the theorem is
true for $n-1$ sets. By induction, there is an SDR $(a_1,\ldots,a_{n-1})$
for $(T_1,\ldots,T_{n-1})$. Thus, the theorem will hold unless $T_n$ is
contained in every SDR for $(T_1,\ldots,T_{n-1})$. Let
\[P^*=\{i:1\le i\le n-1, a_i\hbox{ lies in every SDR for }(T_1,\ldots,T_{n-1})\},\]
and $R^*=\{a_i:i\in P^*\}$. If $T_n\subseteq R^*$, then the sets $T_i$ for
$i\in P^*$ or $i=n$ are $|P^*|+1$ in number but by the lemma contain only
$|P^*|$ elements among them, contrary to hypothesis.
\end{proof}

To deduce the result about cosets, we let $H$ be a subgroup of $G$ of index $n$.
Since the intersection of the conjugates of $H$ is a normal subgroup of finite
index in $G$, by taking the quotient we may assume that $G$ is finite.
Let $L_1,\ldots,L_n$ be the left cosets and $R_1,\ldots,R_n$ the right cosets
of $H$ in $G$; for $i=1,\ldots,n$, let
\[T_i=\{j:L_i\cap R_j\ne\emptyset\}.\]
The union of $k$ left cosets contains $k|H|$ elements, and so cannot be
contained in fewer than $k$ right cosets. So the union of $k$ of the sets $T_i$
has cardinality at least $k$, and Hall's theorem gives a SDR $(a_1,\ldots,a_n)$.
Choosing $g_i\in L_i\cap R_{a_i}$ for $i=1,\ldots,n$ gives the required coset
representatives.

\medskip

The name ``Hall's marriage theorem'' comes from the following interpretation.
Let $\{b_1,\ldots,b_n\}$ be a set of boys, and $S$ a set of girls; let $A_i$
be the set of girls who know $b_i$. Then a necessary and sufficient condition
for it to be possible to marry each boy to a girl he knows is that any $k$
boys know between them at least $k$ girls, for all $k$.

\section{Precursors and equivalents}

In fact, other mathematicians stated equivalent results before Hall.
What  follows is a very brief account; Hall's version became the most
influential, but I leave to a more detailed historical analysis the question
of why this happened. However, the legacy of this activity is shown by the
wide variety of proofs of the theorem which appear in textbooks and expository
accounts.

First, after some definitions, I state some results which are in a sense
equivalent to Hall's theorem.

A \emph{bipartition} of a graph is a partition of the vertex set into two
parts such that each edge is incident with one vertex in each part; a graph
is \emph{bipartite} if it has a bipartition. A \emph{matching} is a set of
pairwise disjoint edges; it is \emph{perfect} if every vertex is incident with
an edge of the matching. A set of vertices is \emph{independent} if it
contains no edge. A \emph{vertex cover} is a set of vertices meeting every
edge.

\begin{theorem}\label{t:equiv}
\begin{enumerate}
\item 
Let $G$ be a graph with bipartition $\{A,B\}$. Suppose that, for
every subset $X$ of $A$, there are at least $|X|$ vertices in $B$ which
have a neighbour in $A$. Then $G$ has a matching which covers every vertex
in $A$.
\item 
Let $M$ be an 
$n\times m$ matrix. A necessary and sufficient condition for the existence
of an injective map $f$ from $\{1,\ldots,n\}$ to $\{1,\ldots,m\}$ such that
$M_{i,f(i)}\ne 0$ for all $i$ is that any $k$ rows together have non-zero
elements in at least $k$ columns, for $0\le k\le n$.
\item
Let $(T_1,\ldots,T_n)$ be a family of subsets of $S$. Suppose that there is a
positive integer $d$ such that, for $1\le k\le n$, any $k$ sets from the family
contain between them at least $k-d$ elements. Then there is a partial SDR
of size $n-d$; that is, elements $a_{i_1},\ldots,a_{i_{n-d}}$ in $S$ such that
$a_j\in T_j$ for $j\in\{i_1,\ldots,i_{n-d}\}$ and $a_j\ne a_k$ for $j\ne k$.
\end{enumerate}
\end{theorem}

\begin{proof}
Parts (i) and (ii) are relatively straightforward variants of Hall's theorem.
For (i), for each $a\in A$, let $T_a$ be  the set of its neighbours in $B$;
the hypotheses in the graph imply Hall's hypotheses, and the conclusion is
then equivalent to Hall's. For (ii), for each row index $i$ of the matrix,
we take $T_i$ to be the set of columns in which the $i$th row has non-zero
entries.

Part (iii) is the \emph{defect form} of Hall's theorem. Though apparently more
general, it is easily proved from the original form, as follows. Take a set
$D$ of $d$ ``dummy'' elements not lying in any of the sets $T_i$, and 
let $T_i^*=T_i\cup D$. The sets $T_i^*$ satisfy Hall's condition, and so have
an SDR; we simply discard the sets whose representative is one of the dummies.
\end{proof}

Part (ii) of the theorem was proved by K\"onig~\cite{konig} and
Egerv\'ary~\cite{egervary} in 1931. Indeed, K\"onig proved the following,
which is also equivalent to Hall's theorem:

\begin{theorem}
The maximum size of a matching in a bipartite graph is equal to the minimum
size of a vertex cover.
\end{theorem}

\begin{proof}
Here is the proof of Hall's theorem (in the form given in
Theorem~\ref{t:equiv}(i)) from K\"onig's. Take a bipartite graph satisfying
the conditions of (ii), and suppose that a vertex cover $C$ contains $t$
vertices in $A$. Then the remaining $|A|-t$ vertices have at least $|A|-t$
neighbours in $B$, all of which must be included in $C$; so $|C|\ge|A|$.
Now K\"onig's theorem shows that there is a matching of size $|A|$, which
must cover all the vertices in $A$, as required.

Here is the proof of K\"onig's theorem from the defect form of Hall's
theorem (Theorem~\ref{t:equiv}(iii) above). Suppose that a minimal vertex cover
contains $k$ vertices in $A$ and $l$ in $B$. Since $A$ is a vertex cover, we
have $k+l\le|A|$; put $d=|A|-(k+l)$. Now take a set $A'$ of $m$ points in $A$,
and suppose that its neighbour set $B'$ in $B$ has $|B'|<m-d$; then
$(A\setminus A')\cup B'$ is a vertex cover with size $|A|-m+|B'|<|A|-d=k+l$,
contrary to assumption. Thus $A'$ has at least $m-d$ neighbours in $B$. By the
defect form of Hall's theorem, there is a matching of size $a-d=k+l$.
\end{proof}

In 1927, Karl Menger~\cite{menger} proved:

\begin{theorem}
In a graph, the maximum number of edge-disjoint paths between two vertices
is equal to the minimum size of an edge cut (a set of edges whose removal
separates these two vertices.
\end{theorem}

Hall's theorem (in the bipartite graph formulation of Theorem~\ref{t:equiv}(i))
can be proved from Menger's: add two new vertices $a$ and $b$ where $a$ is
joined to the vertices in $A$ and $b$ to the vertices in $B$, and applying
Menger's theorem to the vertices $a$ and $b$.

Menger's theorem also has a version for edge-disjoint paths and edge cuts.
From this, it is a small step to the \emph{Max-flow Min-cut Theorem},
for edge-weighted graphs, and so to the \emph{Duality Theorem} of linear
programming.

Other authors who proved similar or equivalent results before Hall include
van der Waerden and Sperner. So it might be said that this was a well-studied
topic at the time.

\section{Dilworth's Theorem and perfect graphs}

In 1950, Dilworth~\cite{dilworth} proved the following theorem. A \emph{chain}
in a partially ordered set is a set of elements any two of which are comparable
(that is, a totally ordered subset); an \emph{antichain} is a set in which any
pair are incomparable.

\begin{theorem}
If the largest antichain in a finite partially ordered set $P$ has cardinality
$m$, then $P$ is the union of $m$ chains.
\end{theorem}

This theorem implies Hall's as follows. With the hypotheses of Hall's theorem,
we partially order $S\cup\{1,\ldots,n\}$ by the rule that, for $a\in S$ and
$i\in\{1,\ldots,n\}$, we put $a<i$ if and only if $a\in T_i$. It is clear that
the largest antichain has size $|S|$, and so we can find $|S|$ chains whose
union is $S\cup\{1,\ldots,n\}$: these must consist of $n$ pairs $(a_i,i)$ where
$(a_1,\ldots,a_n)$ is a SDR, together with the $|S|-n$ unused singletons in $S$.
In fact it is possible to go in the other direction and prove Dilworth's
theorem from Hall's.

The ``dual'' of Dilworth's theorem, stating that if the largest chain has
cardinality $n$, then $P$ is the union of $n$ antichains, is much simpler:
the maximal elements of $P$ form an antichain, and removing them reduces the
size of the largest chain by $1$, so induction finishes the job.

In a poset $P$, the \emph{comparability graph} has as vertices the elements
of $P$, two elements joined if they are comparable; its complement is the
\emph{incomparability graph}.

A finite graph is \emph{perfect} if every induced subgraph has clique number
equal to chromatic number. (Clique number is the size of the largest complete
subgraph, while chromatic number is the smallest number of edgeless induced
subgraphs required to cover the vertices.) Dilworth's theorem asserts that
the incomparability graph of a poset is perfect (since every induced subgraph
of an incomparability graph is an incomparability graph); the easier dual 
asserts that the comparability graph is also perfect.

This led Claude Berge~\cite{berge} to two conjectures, the weak and strong
perfect graph conjectures:
\begin{itemize}
\item The weak conjecture: The complement of a perfect graph is perfect.
\item The strong conjecture: A graph is perfect if and only if it has no
induced subgraph which is a cycle of odd length greater than $3$ or the
complement of one.
\end{itemize}
The weak conjecture was proved by Lov\'asz~\cite{lovasz} in 1972, while the
strong conjecture had to wait for the work of Chudnovsky, Robertson, Seymour
and Thomas~\cite{crst} in 2006.

\section{Matroids}

Matroids were introduced by Hassler Whitney~\cite{whitney}, and independently
by Takeo Nakasawa~\cite{nakasawa}, in 1935, as a model for ``independence''
(which could be linear independence in a vector space, algebraic independence
over a subfield,  acyclicity in a graph, etc.) The definition is simple.
A \emph{matroid} consists of a ground set $E$ together with a non-empty
collection $\mathcal{I}$ of subsets of $E$ called \emph{independent sets},
which is closed downwards (that is, a subset of an independent set is
independent), and has the \emph{exchange property}:
\begin{quote}
If $A,B\in\mathcal{I}$ and $|B|>|A|$, then there exists $b\in B\setminus A$
such that $A\cup\{b\}\in\mathcal{I}$.
\end{quote}
The exchange property implies that any two maximal independent subsets of a
set $X\subseteq E$ have the same cardinality, which is called the \emph{rank}
of $X$, denoted $r(X)$.

Matroids play a central role in combinatorics, with applications ranging from
model theory to stability of frameworks.

Richard Rado generalised Hall's theorem to the setting of matroids. Let
$(T_1,\ldots,T_n)$ be a collection of subsets of $E$. A tuple $(a_1,\ldots,a_n)$
of distinct elements of $E$ is a \emph{system of independent representatives}
for the family if $a_i\in T_i$ for $1\le i\le n$ and
$\{a_1,\ldots,a_n\}\in\mathcal{I}$. Rado~\cite{rado} proved:

\begin{theorem}
Let $(T_1,\ldots,T_n)$ be a collection of subsets of the ground set of a
matroid $(E,\mathcal{I})$. Then the collection has a system of independent
representatives if and only if, for $0\le k\le n$ and any $k$ elements
$i_1,\ldots,i_k\in\{1,\ldots,n\}$, we have
\[r\left(\bigcup_{j=1}^kT_{i_j}\right)\ge k.\]
\end{theorem}

Dominic Welsh~\cite{welsh} showed that many important results in matroid theory,
including Jack Edmonds' \emph{matroid union theorem}, follow from Rado's
theorem.

\section{Counting SDRs}

We begin this section with an important special case of Hall's theorem.

\begin{theorem}\label{t:reg}
Let $(T_1,\ldots,T_n)$ be a family of subsets of $\{1,\ldots,n\}$. Suppose that
there is a positive integer $r$ such that each set $T_i$ has cardinality $r$,
and each point $j\in\{1,\ldots,n\}$ lies in exactly $r$ of these sets. Then
the family has an SDR.
\end{theorem}

\begin{proof}
Choose $k$ of the sets, say $T_{i_1},\ldots,T_{i_k}$. The number of pairs
$(l,j)$ with $l\in\{i_1,\ldots,i_n\}$ and $j\in T_{i_l}$ is $rk$; but each
point $j$ lies in at most $r$ of the sets $\{i_1,\ldots,i_n\}$, and so the
number of such points $j$ is at least $k$. These are precisely the points in
the union of the $k$ sets, so Hall's criterion holds.
\end{proof}

There has been a lot of interest in the question: How many SDRs must such a
family have? I give some results in this section.

A square matrix $M$ is \emph{doubly stochastic} if its elements are all
non-negative and all row and column sums of $M$ are equal to $1$.

\textit{Birkhoff's theorem}~\cite{birkhoff} states:

\begin{theorem}\label{t:hallreg}
The convex hull of the set of $n\times n$ permutation matrices is the set of
$n\times n$ doubly stochastic matrices.
\end{theorem}

\begin{proof}
The fact that a convex combination of permutation matrices is doubly stochastic
is a simple calculation. The fact that every doubly stochastic matrix occurs
can be proved using Hall's Theorem.

The proof is by induction on the number of non-zero entries in $M$. The
smallest possible number is $n$, which occurs only for permutation matrices;
so suppose that the result is proved for matrices with fewer non-zero entries
than $M$. By Hall's theorem (in the matrix form given by
Theorem~\ref{t:equiv}(ii)), there is a permutation $\pi$ such that, for
$1\le i\le n$, the $(i,\pi(i))$ entry of $M$ is positive. Let $\mu$ be the
minimum of all such entries, and $P$ the permutation matrix corresponding to
$\pi$. Subtracting $\mu P$ from $M$ and re-scaling, we obtain a doubly
stochastic matrix with fewer non-zero entries, and the induction goes through.
\end{proof}

There is another aspect which is relevant to us. The \emph{permanent} is
a matrix function similar to the determinant but ``without the signs'':
\[\per(M)=\sum_{\pi\in S_n}\prod_{i=1}^nM_{i,\pi(i)}.\]
Although the determinant can be efficiently computed using linear algebra, the
permanent is intractible to compute.

Suppose that $M$ is a square matrix with all entries $0$ and $1$. Then $M$
represents a bipartite graph on vertices $\{r_i,c_i:1\le i\le n\}$, where $r_i$
and $c_j$ are joined if and only if $M_{i,j}=1$. Then a non-zero term in the
permanent corresponds to a matching in the graph, so $\per(M)$ counts the
matchings. So Theorem~\ref{t:equiv}(ii) gives a necessary and 
sufficient condition for the permanent to be non-zero.

The \emph{van der Waerden conjecture} asserted that the permanent of an
$n\times n$ doubly stochastic matrix is at least $n!/n^n$, with equality
if and only if every entry of the matrix is $1/n$. This was proved 
indepently by Egorychev~\cite{egorychev} and Falikman~\cite{falikman} in 1980.
From it we can deduce a lower bound for the number of matchings in the special
case of Theorem~\ref{t:reg}.
Let $M$ be the matrix defined from the graph by the procedure of the last
paragraph. Then $M$ has row and column sums $r$, so $(1/r)M$ is doubly
stochastic. We conclude that $\per(M)\ge (r/n)^nn!$, and this number is a lower
bound for the number of matchings. We will use this in the next section.

\section{Latin squares}

Let $m,n$ be positive integers with $m\le n$. An $m\times n$ \emph{Latin
rectangle} is an $m\times n$ array with entries from an alphabet of size $n$,
such that each letter occurs once in each row and at most once in each column.
If $m=n$, it is a Latin square. Latin squares first arose in connection with
magic squares, and have applications in universal algebra, graph theory,
statistics, and cryptography.

\begin{prop}
A Latin rectangle can be extended to a Latin square by adding extra rows.
\end{prop}

Clearly it suffices to show that, if $m<n$, then one more row can be added.
The letters $(a_1,\ldots,a_n)$ in
the added row must be all distinct, and must not occur in the $m$ elements in
their column in the given rectangle; so they must form an SDR for the sets
$T_1,\ldots,T_n$, where $T_i$ is the set of letters not appearing in the $i$th
column. Each set $T_i$ has cardinality $n-m$, and each letter lies in $n-m$
of these sets (those corresponding to columns not containing that letter).
By Theorem~\ref{t:reg}, the SDR exists (and, indeed, there are many SDRs).

\medskip

Using the result of the last section, we can obtain a lower bound for the
number of Latin squares. The number of possible $(m+1)$st rows is at least
$((n-m)/n)^n\,n!$; so the number of Latin squares is at least
\[\prod_{m=0}^{n-1}\left(\frac{n-m}{n}\right)^n(n!)=\frac{(n!)^{2n}}{n^{n^2}}.\]
Using Stirling's formula, this is roughly $(n/\mathrm{e}^2)^{n^2}$, so 
asymptotic in the logarithm to the total number of $n\times n$ arrays over an
alphabet of size $n$.

A variant of Latin squares used in statistics are the designs known as
\emph{Youden squares}~\cite{youden,fisher}. In one formulation, a Youden square
is a Latin rectangle whose columns are the blocks of a symmetric balanced
incomplete-block design on the set of letters. The existence of Youden squares 
is an consequence of Theorem~\ref{t:reg}, and the enumeration results
above give estimates for the number of such designs. Indeed, the design needs
to be equireplicate but balance is not required.

\section{Algorithmic version}

Hall's theorem guarantees the existence of a system of distinct representatives.
But the condition itself appears to involve a lot of checking, since we need
to consider every subset of the index set.

The proof of the Max-Flow Min-Cut theorem by Ford and Fulkerson~\cite{ff} was
used by Dinic~\cite{dinic} and Edmonds and Karp~\cite{ek} to give an
efficient algorithm. In the situation of Hall's theorem, it
starts with a partial SDR, and returns either a larger partial SDR or a
failure of Hall's condition. Each round of the algorithm runs in polynomial
time, and it only has to be repeated at most $n$ times. 

However, Dima Fon-Der-Flaass~\cite{fdf} showed that the problem of finding a 
$2$-dimensional array of distinct representatives is NP-hard. (In this problem
we are given a rectangular array of sets and asked to find representatives
such that the representatives in any row or column of the array are distinct.)

\section{The infinite}

Hall's theorem fails in the infinite case. Marshall Hall Jr.\ (no relation)
gave the following example: take $S$ to be the set of positive integers,
$T_0=S$ and $T_i=\{i\}$ for $i>0$. Each set $T_i$ for $i>0$ must be 
represented by its unique element, leaving no representative for $T_0$.
On the other hand, any $k$ of the sets contain either $k$ or infinitely many
elements between them, depending on whether $T_0$ is included.

Marshall Hall~\cite[p.~51]{hall_m} proved a version of
Hall's Theorem for the infinite case:

\begin{theorem}
Let $(T_i:i\in I)$ be a family of finite subsets of an infinite set $S$.
Then the family has a system of distinct representatives if and only if, for
every natural number $k$, any $k$ of the sets contain at least $k$ elements
between them.
\end{theorem}

A recent application by Downarowicz, Huczek and Zhang~\cite{dhz} uses Hall's
theorem to show that any countable amenable group $G$ has a tiling into finite
tiles of only finitely many distinct shapes, where the tiles are almost
invariant under any given finite subset of $G$. (I am grateful to
Josh Frisch for this information.) They use a slightly different infinite
version of Hall's theorem: they assume a countable set $\mathcal{T}$ of
subsets of a countable set $X$ such that there exists a positive integer $N$
with the property that each point of $X$ lies in at most $N$ members of
$\mathcal{T}$, while each set in $\mathcal{T}$ contains at least $N$ members
of $X$.

A general necessary and sufficient condition for an arbitrary family of sets
to have a transversal was found by Aharoni, Nsh-Williams and Shelah~\cite{anws}.
I will not describe the precise result here.

\section{A final variant}

I will finish with a generalization of Hall's theorem due to Aharoni and
Haxell~\cite{ah}. This uses the notion of a \emph{hypergraph}, a structure
consisting of a set of vertices and a collection of subsets called
\emph{edges} (or sometimes \emph{hyperedges}). The authors define a
\emph{system of distinct representatives} for a family
$\mathcal{A}=\{H_1,\ldots,H_m\}$ of hpergraphs to be a function $f$ which
selects an edge $f(H_i)$ of each hypergraph $H_i$ in such a way that
$f(H_i)\in H_i$ for $i=1,\ldots,m$.

A \emph{matching} is a set of pairwise disjoint edges. 
A set $F$ of edges is said to be \emph{pinned} by a set $K$ of edges if
every edge in $F$ is met by some edge in $K$.

\begin{theorem}
A sufficient condition for a family $\mathcal{A}$ of hypergraphs to have an SDR
is that, for every subfamily $\mathcal{B}$ of $\mathcal{A}$, there is a
matching $M$ in $\bigcup\mathcal{B}$ which cannot be pinned by fewer than
$|\mathcal{B}|$ disjoint edges from $\bigcup\mathcal{B}$.
\end{theorem}

If each edge in each hypergraph is a singleton, the hypothesis reduces to that
of Hall's theorem (each hypergraph can be represented as a set, the union of
its singleton edges), and the conclusion is the same as Hall's.

A feature of the proof is that it is topological: it uses Sperner's
Lemma~\cite{sperner}.

\paragraph{Acknowledgement}
I am grateful to several people including John Britnell, Josh Frisch, Scott
Harper, and two reviewers for very helpful comments which have improved the 
paper.

\end{document}